\documentclass[11pt]{elsarticle}
\usepackage[utf8]{inputenc}
\usepackage{graphicx}
\graphicspath{ {./images/} }
\usepackage{mathtools,esint}
\usepackage{wasysym}
\allowdisplaybreaks

\usepackage{srcltx}
\usepackage{eurosym}
\usepackage{amssymb}
\usepackage{amsfonts}
\usepackage[all,arc]{xy}
\usepackage{enumerate}
\usepackage{mathrsfs}
\usepackage{pst-node}
\usepackage{amsmath,amsthm}
\usepackage[T1]{fontenc}
\usepackage[titletoc,toc,page]{appendix}

\usepackage[colorlinks,plainpages=true,pdfpagelabels,hypertexnames=true,colorlinks=true,pdfstartview=FitV,linkcolor=blue,citecolor=green,urlcolor=black,pagebackref]{hyperref}
\PassOptionsToPackage{unicode}{hyperref}
\PassOptionsToPackage{naturalnames}{hyperref}

\newcommand*\bb{\mathbb}

\newcommand *\w{^\wedge}

\newcommand*\de{\partial}

\makeatletter
\newcommand{\vo}{\vec{o}\@ifnextchar{^}{\,}{}}
\makeatother
\def\YYint#1#2#3{{\setbox0=\hbox{$#1{#2#3}{\iint}$}
    \vcenter{\hbox{$#2#3$}}\kern-.50\wd0}}
\def\XXint#1#2#3{{\setbox0=\hbox{$#1{#2#3}{\int}$}
    \vcenter{\hbox{$#2#3$}}\kern-.50\wd0}}

\makeatletter
\def\namedlabel#1#2{\begingroup
   \def\@currentlabel{#2}%
   \label{#1}\endgroup
}
\makeatother
\makeatletter
\newcommand{\rmh}[1]{\mathpalette{\raisem@th{#1}}}
\newcommand{\raisem@th}[3]{\hspace*{-1pt}\raisebox{#1}{$#2#3$}}
\makeatother

\newcommand{\descref}[2]{\hyperref[#1]{\textnormal{\textcolor{black}{(}\textcolor{blue}{\bf #2}\textcolor{black}{)}}}}
\newcommand{\dref}[2]{\hyperref[#1]{\textcolor{black}{(}\textcolor{blue}{\bf #2}\textcolor{black}{)}}}
\newcommand\RR{\mathbb{R}}

\newcommand{\al}{\alpha}

\newcommand{\ep}{\epsilon}

\newcommand{\Om}{\Omega}

\DeclareMathOperator{\spt}{spt}

\DeclareMathOperator*{\esssup}{ess\,sup}
\DeclareMathOperator*{\essinf}{ess\,inf}

\usepackage[ddmmyyyy]{datetime}
\usepackage[margin=2cm]{geometry}
\makeatletter
\g@addto@macro\normalsize{%
  \setlength\abovedisplayskip{2pt}
  \setlength\belowdisplayskip{2pt}
  \setlength\abovedisplayshortskip{4pt}
  \setlength\belowdisplayshortskip{4pt}
}
\numberwithin{equation}{section}
\everymath{\displaystyle}
\usepackage[capitalize,nameinlink]{cleveref}
\crefname{section}{Section}{Sections}
\crefname{subsection}{Subsection}{Subsections}
\crefname{condition}{Condition}{Conditions}
\crefname{hypothesis}{Hypothesis}{Hypothesis}
\crefname{assumption}{Assumption}{Assumptions}
\crefname{lemma}{Lemma}{Lemmas}
\crefname{claim}{Claim}{Claims}
\crefname{remark}{Remark}{Remarks}

\crefformat{equation}{\textup{#2(#1)#3}}
\crefrangeformat{equation}{\textup{#3(#1)#4--#5(#2)#6}}
\crefmultiformat{equation}{\textup{#2(#1)#3}}{ and \textup{#2(#1)#3}}
{, \textup{#2(#1)#3}}{, and \textup{#2(#1)#3}}
\crefrangemultiformat{equation}{\textup{#3(#1)#4--#5(#2)#6}}%
{ and \textup{#3(#1)#4--#5(#2)#6}}{, \textup{#3(#1)#4--#5(#2)#6}}%
{, and \textup{#3(#1)#4--#5(#2)#6}}

\Crefformat{equation}{#2Equation~\textup{(#1)}#3}
\Crefrangeformat{equation}{Equations~\textup{#3(#1)#4--#5(#2)#6}}
\Crefmultiformat{equation}{Equations~\textup{#2(#1)#3}}{ and \textup{#2(#1)#3}}
{, \textup{#2(#1)#3}}{, and \textup{#2(#1)#3}}
\Crefrangemultiformat{equation}{Equations~\textup{#3(#1)#4--#5(#2)#6}}%
{ and \textup{#3(#1)#4--#5(#2)#6}}{, \textup{#3(#1)#4--#5(#2)#6}}%
{, and \textup{#3(#1)#4--#5(#2)#6}}

\crefdefaultlabelformat{#2\textup{#1}#3}

\newtheorem{theorem}{Theorem}[section]
\newtheorem{lemma}[theorem]{Lemma}

\newtheorem{proposition}[theorem]{Proposition}

\newtheorem{definition}[theorem]{Definition}% Use {\rm ...}
\newtheorem{remark}[theorem]{Remark}        % Use {\rm ...}

\numberwithin{equation}{section}

\usepackage{enumitem}
\newlist{steps}{enumerate}{1}
\setlist[steps, 1]{label = \textcolor{Cerulean}{Step \arabic*:}}
\makeatletter
\def\ps@pprintTitle{%
	\let\@oddhead\@empty
	\let\@evenhead\@empty
	\def\@oddfoot{}%
	\let\@evenfoot\@oddfoot}
\makeatother
\DeclarePairedDelimiterX{\inp}[2]{\langle}{\rangle}{#1, #2}
\newcommand{\norm}[1]{\left\lVert#1\right\rVert}
\let\oldint\int
\renewcommand{\int}{\oldint\limits}
\let\oldfint\fint
\renewcommand{\fint}{\oldfint\limits}
\let\oldiint\iint
\renewcommand{\iint}{\oldiint\limits}

\let\oldiiint\iiint
\renewcommand{\iiint}{\oldiiint\limits}
\usepackage{doi}
\definecolor{aorta}{rgb}{0.0, 0.5, 0.0}
\definecolor{darklavender}{rgb}{0.45, 0.31, 0.59}
\AtBeginDocument{%
\hypersetup{citecolor=red}}

\begin{document}

\begin{frontmatter}
\title{On The Weak Harnack Estimate For Nonlocal Equations}
\author[myaddress]{Harsh Prasad\tnoteref{thanks}}
\ead{harsh@tifrbng.res.in}

\tnotetext[thanks]{Supported by the Department of Atomic Energy,  Government of India, under	project no.  12-R\&D-TFR-5.01-0520}

\address[myaddress]{Tata Institute of Fundamental Research, Centre for Applicable Mathematics, Bangalore, Karnataka, 560065, India}

\begin{abstract}
We prove a weak Harnack estimate for a class of doubly nonlinear nonlocal equations modelled on the nonlocal Trudinger equation 
\begin{align*}
	\de_t(|u|^{p-2}u) + (-\Delta_p)^s u = 0
\end{align*} 
for $p\in (1,\infty)$ and $s \in (0,1)$. Our proof relies on expansion of positivity arguments developed by DiBenedetto, Gianazza and Vespri adapted to the nonlocal setup. Even in the linear case of the nonlocal heat equation and in the time-independent case of fractional $p-$Laplace equation, our approach provides an alternate route to Harnack estimates without using Moser iteration, log estimates or Krylov-Safanov covering arguments. 
\end{abstract}

\end{frontmatter}
%\begin{singlespace}
\tableofcontents
%\end{singlespace}

\section{Introduction}
\subsection{The Problem}
We prove weak Harnack estimates for positive super-solutions to the nonlocal Trudinger equation  
\begin{equation}\label{eq:main-eq}
    \de_t(|u|^{p-2}u) - Lu = 0 \text{ weakly in } E_T
\end{equation}
where $E_T = E \times (-T,T]$ for some open set $E \subset \bb{R}^N$ and some $T> 0$; $L$ is a nonlocal operator modelled on the fractional $p-$Laplacian given by
\begin{equation}
    Lu(x,t) = \text{P.V.}\int\limits_{\RR^N} K(x,y,t)|u(x,t)-u(y,t)|^{p-2}(u(x,t)-u(y,t))\,dy 
\end{equation}
where the kernel $K:(-T,T]\times\bb{R}^N\times\bb{R}^N \rightarrow [0,\infty)$ is a measurable function satisfying 
\begin{equation}
    K(x,y,t) = K(y,x,t)
\end{equation}
and 
\begin{equation}
    \lambda|x-y|^{-N-sp} \leq K(x,y,t) \leq \Lambda |x-y|^{-N-sp}
\end{equation}
for some $0<\lambda \leq \Lambda$, $p \in (1,\infty)$ and $s\in (0,1)$. The general equation is modelled after the nonlocal Trudinger equation
\begin{equation}    
  \de_t(|u|^{p-2}u) + (-\Delta_p)^s u = 0. 
\end{equation}
In particular for the case $p=2$, we get the nonlocal heat equation and in the time-independent case we get the fractional $p-$Laplacian. 
\subsection{Context}
For the linear setup where $p=2$, the weak Harnack was established in \cite{Felsinger2013} using Moser iteration and log estimates for globally non-negative solutions. It was extended to locally non-negative solutions with a tail contribution in \cite{Strmqvist2019} again using Moser iteration and log estimates. However, the results in \cite{Strmqvist2019} were applicable only to global solutions. A Harnack for local solutions was obtained in \cite{kassmann2023} as a consequence of the weak Harnack proved in \cite{Felsinger2013} and an improved boundedness estimate with an $L^1$ tail term. \\
As we can see, the various proofs of weak Harnack estimate rely on Moser iteration and log estimates. In fact, even in the elliptic case \cite{DiCastro2014}, the proof relies on log estimates and Krylov-Safanov covering arguments. The proof we present here is based on the expansion of positivity method pioneered in \cite{actaharnack,dukeharnack,anallisnsharnack,DiBenedetto2012} adapted to the nonlocal setting. It 
works for the doubly nonlinear nonlocal equation \cref{eq:main-eq}, the nonlocal heat equation and in the time independent case of the fractional $p-$Laplace equation equally well without the use of Moser iteration, log estimates or any form of covering arguments. It serves as a proof of concept to a more robust approach to proving Harnack estimates which can handle nonlinear nonlocal parabolic equations. \\
We now discuss some prior developments below. 
\subsubsection{Local Case}
The local Trudinger equation 
\begin{equation}\label{eq:local}
\de_t(|u|^{p-2}u) - \Delta_p u = 0
\end{equation}
is an example of a general class of doubly nonlinear parabolic equations which arise in the study of shallow water flows, glacier dynamics and gas flow under friction among other physical models \cite{Mahaffy1976,ALONSO2008,Leugering2018}.
From the point of view of De Giorgi-Nash-Moser theory \cref{eq:local} serves as an important prototype for equations without translation invariance. Harnack estimates were proved in \cite{Trudinger1968,gianazza2006harnack,Kinnunen2006} and H\"older regularity was proved in \cite{Kuusi2012,Kuusi2011,Bgelein2021}. Other wrorks of interest in this direction are \cite{ivanov1992quasilinear,ivanov1994holder,ivanov1991regularity,vespri1992local,vespri2022extensive,liaopart2,liaopart3}.
\subsubsection{Nonlocal Case} 
The nonlocal term models many physical processes such as surface diffusion, fracture dynamics in the plane and phase transition with long range interactions \cite{lions1972inequations,giacomin1998deterministic,Caffarelli2012}. The De Giorgi-Nash-Moser theory was extended to nonlocal elliptic equations in \cite{Kassmann2008,dyda2020regularity,DiCastro2014,DiCastro2016,Cozzi2017} whereas for nonlocal parabolic equations it was extended in \cite{caffarelli2011regularity,Felsinger2013,Strmqvist2019,Strmqvist2019bdd}

As in the local case, \cref{eq:main-eq} is an important prototype of a nonlocal equation without translation invariance. For the linear setup where $p=2$, the Harnack estimate was established in \cite{Felsinger2013,Strmqvist2019,kassmann2023}.  For the nonlocal Trudinger equation \cite{Banerjee2021} provides local boundedness for positive sub-solutions and reverse H\"older estimates for positive super-solutions. 
\subsection{Notation and Definition}
\subsubsection{Notation}
We begin by collecting the standard notation that will be used throughout the paper:
\begin{itemize}
\item  $N$ refers to the space dimension. 
\item  $\partial_t f$ refers to the time derivative of $f$.
\item  $E$ denotes a fixed open set in $\mathbb{R}^N$ with boundary $\partial E$. 
\item For $T>0$,  let $E_T\coloneqq E\times (-T,T]$. This will be our reference domain.
\item The parabolic boundary $\de_p E_{T} \coloneqq \de E_T - \Bar{E}\times\{T\}$. We define the parabolic boundary for general open sets in a similar manner. 
\item For $\rho>0$ and $x_0 \in \bb{R}^N$ we define 
\begin{align*}
	&K_{\rho}(x_0)=\{x\in\RR^N:|x-x_0|<\rho\}.
\end{align*} 
When $x_0 = 0$ we write $K_{\rho}(0)$ as $K_{\rho}$. 
\item For a scaling parameter $\theta >0$ we write
\[
(x_0,t_0) + Q_{\rho}(\theta) = K_{\rho}(x_0)\times(t_0-\theta\rho^{sp},t_0].
\]
\item Integration with respect to either space or time only will be denoted by a single integral $\int$ whereas integration on $\Om\times\Om$ or $\RR^N\times\RR^N$ will be denoted by a double integral $\iint$.
\item  A universal constant is a positive constant which depends only on the dimension $N$, the integrability $p$, the differentiability $s$, and the ellipticity constants $\lambda$ and $\Lambda$. 
\end{itemize}

\subsubsection{Function spaces}
Let $1<p<\infty$, we denote by $p'=p/(p-1)$ the conjugate exponent of $p$. Let $\Om$ be an open subset of $\RR^N$. We define the {\it Sobolev-Slobodeki\u i} space, which is the fractional analogue of Sobolev spaces.
\begin{align*}
    W^{s,p}(\Om)=\left\{ \psi\in L^p(\Omega): [\psi]_{W^{s,p}(\Om)}<\infty \right\}, s\in (0,1),
\end{align*} where the seminorm $[\cdot]_{W^{s,p}(\Om)}$ is defined by 
\begin{align*}
    [\psi]_{W^{s,p}(\Om)}=\left(\iint\limits_{\Om\times\Om} \frac{|\psi(x)-\psi(y)|^p}{|x-y|^{N+ps}}\,dx\,dy\right)^{\frac 1p}.
\end{align*}
The space when endowed with the norm $\norm{\psi}_{W^{s,p}(\Om)}=\norm{\psi}_{L^p(\Om)}+[\psi]_{W^{s,p}(\Om)}$ becomes a Banach space. The space $W^{s,p}_0(\Om)$ is the subspace of $W^{s,p}(\RR^N)$ consisting of functions that vanish outside $\Om$. We also define the tail space as 
\begin{align}\label{tailspace}
    L^m_{\alpha}(\RR^N):=\left\{ v\in L^m_{\text{loc}}(\RR^N):\int\limits_{\RR^N}\frac{|v(x)|^m}{1+|x|^{N+\alpha}}\,dx<+\infty \right\},\,m>0,\,\alpha>0.
\end{align}

\subsubsection{Definition of weak solution}
We work with the following notion of a solution. 
\begin{definition}
    A function $u\in L_{\text{loc}}^p(-T,T;W^{s,p}_{\text{loc}}(E))\cap L_{\text{loc}}^\infty(-T,T;L^2_{\text{loc}}(E))\cap L^\infty(-T,T;L^{p-1}_{sp}(\RR^N))$ is said to be a weak sub(super)-solution to \cref{eq:main-eq} if for every compact subset $\mathcal{K} \subset E$ and every closed subinterval $[t_1,t_2]\subset (-T,T]$, it holds that
    \begin{align*}
        \int\limits_{\mathcal{K}} |u(x,t)|^{p-2}u(x,t)\phi(x,t)\,dx\Big|_{t_1}^{t_2} &- \int_{t_1}^{t_2}\int\limits_{\mathcal{K}} |u(x,t)|^{p-2}u(x,t)\partial_t\phi(x,t)\,dx\,dt\\
        &+\int_{t_1}^{t_2}\iint\limits_{\bb{R}^N\times\bb{R}^N}\,K(x,y,t)|u(x,t)-u(y,t)|^{p-2}(u(x,t)-u(y,t))(\phi(x,t)-\phi(y,t))\,dy\,dx\,dt\\
        &\leq(\geq) 0,
    \end{align*} for all non-negative test functions $\phi\in L_{\text{loc}}^p(I,W_0^{s,p}(\mathcal{K}))\cap W_{\text{loc}}^{1,p}(I,L^p(\mathcal{K}))$. 
\end{definition}
\subsection{Main Theorem}
We state our main theorem below. 
\begin{theorem}
    Let $u$ be a globally non-negative weak supersolution to \cref{eq:main-eq}. Then there exists a scale $\delta_0 \in (0,1)$ an integrability exponent $0<q<1$ and a constant $H >1$ all depending only on the data such that
    \[
         \left(\fint_{K_{\rho}(y)}u^q(\cdot,\tau)\,dx\right)^{\frac{1}{q}} \leq H \text{  } \underset{K_{2\rho(y)}}{\essinf} u(\cdot,t)  \quad \text{for all times} \quad t \in \left(\tau + \delta_0(\rho/2)^{sp},\tau + \delta_0\rho^{sp}\right).
    \] 
\end{theorem}
\iffalse
\subsubsection{Outline of Proof}
The main tools used in our proof are expansion of positivity and a quantitative Severini-Egorov result. Roughly, expansion of positivity says that if we have a positive supersolution that is bounded from below in a good chunk of a ball at some time then upto a universal factor, it remains bounded from below in a bigger ball at some later times. On the other hand, the  quantitative Severini-Egorov asserts that if we have a fractional Sobolov function which is bounded away from zero in a good chuck of a ball then we can find a point in the ball and a universal neighborhood of the point where the function remains large in a good chunk of the neighborhood. \\
These two tools can be combined to yield another expansion of positivity result which removes the polynomial dependence of a measure theoretic parameter on the time scale. If we do not use the quantitative Severini-Egorov it may still be possible to extract a weak Harnack estimate albeit one whose time scale may be intrinsic to the estimate. This seems to be a peculiar feature of nonlocal equations - see \cref{sec:clust} and \cref{rem:polydep}.\\
The second expansion of positivity provides time scales and reduction parameters that are universal along with a polynomial dependence on a measure theoretic parameter. Then a simple integration using the layer cake representation finishes the proof.  
\fi
\section{Auxiliary Results}
\subsection{Technical Lemmas}
For $k, w\in\bb{R}$ we define 
\begin{equation*}
	\mathfrak g_\pm (w,k)=\pm (p-1)\int_{k}^{w}|s|^{p-2}(s-k)_\pm\,ds.
\end{equation*}
Note that $\mathfrak g_\pm (w,k)\ge 0$.
The following lemma can be found in Lemma 2.2 \cite[Lemma 2.2]{ACERBI1989115} for $0<\al<1$ and in  in \cite[estimate (2.4)]{Giaquinta1986} for $\al > 1$. 
\begin{lemma}\label{lem:fusco}
    For any $\alpha>0$ there exists a constant $C=C(\al)$ such that for every $a,b \in \bb{R}$ we have
    \[
    \frac{1}{C}||b|^{\al-1}b-|a|^{\al-1}a| \leq (|a|+|b|)^{\al-1}|b-a| \leq C||b|^{\al-1}b-|a|^{\al-1}a|
    \]
    where we set $|c|^{\al-1}c = 0$ for $c=0$. 
\end{lemma}
Using \cref{lem:fusco} we get the following as in \cite[Lemma 2.2]{Bgelein2021}
\begin{lemma}\label{lem:g}
There exists a constant $\boldsymbol\gamma=\boldsymbol\gamma(p)$ such that,
for all $w,k\in\bb{R}$, the following inequality holds true:
\begin{align*}
	\tfrac1{\boldsymbol\gamma} \big(|w| + |k|\big)^{p-2}(w-k)_\pm^2
	\le
	\mathfrak g_\pm (w,k)
	\le
	\boldsymbol\gamma \big(|w| + |k|\big)^{p-2}(w-k)_\pm^2
\end{align*}
\end{lemma}
Next, we have the following Sobolov embedding result from \cite[Lemma 2.3]{Ding2021}
\begin{lemma}\label{fracpoin}
    Let $t_2>t_1>0$ and suppose $s\in(0,1)$ and $1\leq p<\infty$. Then for any $f\in L^p(t_1,t_2;W^{s,p}(B_r))\cap L^\infty(t_1,t_2;L^2(B_r))$, we have
    \begin{equation*}%\label{sobolev}
    	\begin{array}{rcl}
        \int_{t_1}^{t_2}\int_{B_r}|f(x,t)|^{p\left(1+\frac{2s}{N}\right)}\,dx\,dt
        & \leq &  C(N,s,p)\left(r^{sp}\int_{t_1}^{t_2}\int_{B_r}\int_{B_r}\frac{|f(x,t)-f(y,t)|^p}{|x-y|^{N+sp}}\,dx\,dy\,dt+\int_{t_1}^{t_2}\int_{B_r}|f(x,t)|^p\,dx\,dt\right) \\
        &&\quad \times\left(\sup_{t_1<t<t_2}\fint_{B_r}|f(x,t)|^2\,dx\right)^{\frac{sp}{N}}.
        \end{array}
    \end{equation*}  
\end{lemma}
Finally, we have the following iteration result from \cite[Chapter I, Lemma 4.1]{DiBenedetto1993} 
\begin{lemma}\label{geo_con}
	Let $\{Y_n\}$, $n=0,1,2,\ldots$, be a sequence of positive number, satisfying the recursive inequalities 
	\[ Y_{n+1} \leq C b^n Y_{n}^{1+\alpha},\]
	where $C > 1$, $b>1$, and $\alpha > 0$ are given numbers. If 
	\[ Y_0 \leq  C^{-\frac{1}{\alpha}}b^{-\frac{1}{\alpha^2}},\]
	then $\{Y_n\}$ converges to zero as $n\to \infty$. 
\end{lemma}

\subsection{Shrinking Lemma}
As we are working with nonlocal equations, we do not have any bounds on the gradient of the solution and so we cannot apply the usual De Giorgi isoperimetric inequality to control jumps for the solutions. In fact, such an inequality cannot be true due to the presence of a jump term (see \cite[Theorem 4]{adimurthi2023holder}.) On the other hand, it was observed in \cite{aptHolder} that the isoperimetric inequality for nonlocal equations is a property of the solution itself and can be obtained by a more careful energy estimate. The following can be found in \cite[Lemma 3.4]{aptHolder}
\begin{lemma}\label{lem:shrinking}
Suppose that for some level $m$, some fraction $f \in (0,1)$ and all time levels $\tau$ in some interval $J$ we have
\[
|[u(\cdot,\tau)>m]\cap K_{\rho}| \geq f|K_{\rho}|
\]
and we can arrange that
\begin{align}\label{smallness}
\int_J\int_{K_{\rho}} (u-l)_{-}(x,t)\int_{K_{A\rho}}\frac{(u-l)_{+}^{p-1}(y,t)}{|x-y|^{N+sp}}\,dy\,dx\,dt \leq C_1\frac{l^p}{\rho^{sp}}|Q|
\end{align}
where 
\[
l = \frac{m}{2^{j}} \qquad j\geq 1 \qquad \text{ and } \qquad Q = K_{\rho} \times J
\] 
then we have
\[
\left|\left[u<\frac{m}{2^{j+1}}\right]\cap Q\right| \leq \frac{1}{f}\left(\frac{C}{2^{j}-1}\right)^{p-1}|Q|.
\]
where $C>0$ depends only on $C_1,A$ and $N$.
\end{lemma}
\subsection{Clustering Lemma}\label{sec:clust}
The next lemma is a quantitative version of the Severini-Egorov theorem \cite{severini1910sulle,egorov1911fonctions} which was first obtained for Sobolov functions in \cite{clusterarma,clusteringl1}. It was proved by  for $W^{s,p}$ functions in \cite{düzgün2023clustering}. Its main use is in removing the power like dependence on $\delta$ in the expansion of positivity. Usually, in the local case, this lemma is used to also remove the exponential dependence on $\al$ and turn it into a polynomial decay in the expansion of positivity \cite{DiBenedetto2012}; but for us, our shrinking lemma already gives us polynomial dependence on $\al$. In particular, we may still get a weak Harnack without using the following lemma but at the cost that our weak Harnack may depend on $\delta$ i.e. the scale of the time interval considered.   
\begin{lemma}
    Let $u \in W^{s,p}(K_{\rho})$ satisfy
    \[
    |\{u>1\}\cap K_{\rho}| > \alpha |K_{\rho}| \quad \text{and} \quad [u]_{s,p;K_{\rho}} \leq \gamma\rho^{\frac{N-sp}{p}}
    \]
    for some $\gamma > 0$ and $0<\alpha<1$. Then for every $\delta \in (0,1)$ and $0<\lambda<1$ there exists a $y \in K_{\rho}$ and $\sigma = \sigma(\alpha,\delta.\gamma,\lambda,N) \in (0,1)$ such that 
    \[
    |\{u>\lambda\} \cap K_{\sigma\rho}(y)| > (1-\delta)|K_{\sigma\rho}|.
    \]
    Further the reduction parameter $\sigma$ depends on the parameter $\alpha$ and on the constant $\gamma$ as
    \begin{equation}
        \sigma = B^{-1}\frac{\alpha^{p+1}}{\gamma^p}
    \end{equation}
    for a constant $B = B(\delta,\lambda,N)>1$. 
\end{lemma}

\section{Preliminary Estimates}
\subsection{Energy Estimates}
The energy estimate below follows from similar computations already present in the literature - for example we can handle the time term as in \cite[Proposition 3.1]{Bgelein2021} and the space term as in \cite[Theorem 3.1]{aptHolder} and \cite[Proposition 2.1]{liaoholder}. Indeed, the estimate is already available in \cite[Lemma 3.1]{Banerjee2022} except the second term on the left hand side which is crucial to the usage of the shrinking lemma - it can be obtained from the nonlocal term exactly as in \cite{aptHolder,liaoholder}.
\begin{proposition}\label{Prop:energy}
	Let $u$ be a  local weak sub(super)-solution to in $E_T$.
	There exists a constant $\boldsymbol \gamma >0$ depending only on the data such that
 	for all cylinders $Q_{R,S}=K_R(x_o)\times (t_o-S,t_o)\Subset E_T$,
 	every $k\in\bb{R}$, and every non-negative, piecewise smooth cutoff function
 	$\zeta$ vanishing on $\de K_{R}(x_o)\times (t_o-S,t_o)$,  there holds
\begin{align*}
	\esssup_{t_o-S<t<t_o}&\int_{K_R(x_o)\times\{t\}}	
	\zeta^p\mathfrak g_\pm (u,k)\,dx +
	\iint\limits_{Q_{R,S}} (u-k)_{\pm}(x,t)\int\limits_{K_R(x_o)}\frac{(u-k)_{\mp}^{p-1}(y,t)}{|x-y|^{N+sp}}\,dy\,dx\,dt 
	\\
	&\qquad+\underset{Q_{R,S}}{\iint}|(u-k)_{\pm}(x,t)\zeta(x,t)-(u-k)_{\pm}(y,t)\zeta(y,t)|^p\,d\mu\,dt\\
	&\leq
	\boldsymbol \gamma\int_I\iint\limits_{K\times K} \max\{(u-k)_{\pm}(x,t),(u-k)_{\pm}(y,t)\}^{p}|\zeta(x,t)-\zeta(y,t)|^p\,d\mu\,dt
	+
	\iint_{Q_{R,S}}\mathfrak g_\pm (u,k)|\partial_t\zeta^p| \,dx\,dt
	\\
	&\qquad+\int_{K_R(x_o)\times \{t_o-S\}} \zeta^p \mathfrak g_\pm (u,k)\,dx 
	+ C\underset{\stackrel{t \in I}{x\in \text{supp }\xi}}{\esssup}\,\int_{K^c}\frac{w_{\pm}^{p-1}(y,t)}{|x-y|^{N+sp}}\,dy\iint\limits_{I\times K} w_{\pm}(x,t)\xi(x,t)\,dx\,dt
\end{align*}
\end{proposition}

\subsection{Tail Estimates}\label{sec:tail}

In this section, we shall outline how the  estimate for the tail term is made and we shall refer to this section whenever a similar calculation is required in subsequent sections.

For any level $k$ we want to estimate
\[
\underset{\stackrel{t \in J;}{x\in \spt\zeta}}{\esssup}\int_{K_{\rho}^c(y_0)}\frac{(u-k)_{-}^{p-1}(y,t)}{|x-y|^{N+sp}}\,dy ,
\]
where $J$ is some time interval and $\zeta$ is a cutoff function. We will typically choose $\zeta$ to be supported in $K_{\vartheta\rho}$ for some $\vartheta \in (0,1)$. In such cases,
\[
|y-y_0| \leq |x-y|\left(1+\frac{|x-y_0|}{|x-y|}\right)\leq  |x-y|\left(1+\frac{\vartheta}{(1-\vartheta)}\right),
\]
so we can make a first estimate
\[
\underset{\stackrel{t \in J;}{ x\in \spt\zeta}}{\esssup}\int_{K_{\rho}^c(y_0)}\frac{(u-k)_{-}^{p-1}(y,t)}{|x-y|^{N+sp}}\,dy 
\leq \frac{1}{(1-\vartheta)^{N+sp}}\underset{t \in J}{\esssup}\int_{K_{\rho}^c(y_0)}\frac{(u-k)_{-}^{p-1}(y,t)}{|y-y_0|^{N+sp}}\,dy. 
\]
At this stage we invoke the global nonnegativity assumption on $u$ to get $(u-k)_{-} \leq k$. This leads us to the following  estimate:
\[
\underset{\stackrel{t \in J;}{ x\in \spt\zeta}}{\esssup}\int_{K_{\rho}^c(y_0)}\frac{(u-k)_{-}^{p-1}(y,t)}{|x-y|^{N+sp}}\,dy 
\leq \frac{C}{(1-\vartheta)^{N+sp}}\frac{k^{p-1}}{\rho^{sp}}
\]
for a universal constant $C>0$. In the absence of global non-negativity, the situation is more complicated and leads to tail alternatives.  

\section{Expansion of Positivity}
Throughout this section, we will work with globally nonnegative super-solutions to ~\cref{eq:main-eq}. We assume $(x_o,t_o) \in E$ and 
\begin{equation}\label{Eq:1:5}
	K_{8\varrho}(x_o)\times\big(t_o,t_o+(8\varrho)^{sp}\big]\subset E_T.
\end{equation}
We will prove the following expansion of positivity result for super-solutions. 
\begin{proposition}\label{Prop:1:1}
	Let $u$ be a globally nonnegative weak super-solution to \cref{eq:main-eq} in $E_T$.
	Suppose for some $(x_o,t_o)\in E_T$, $M>0$, $\alpha \in(0,1)$ and $\varrho>0$ we have \cref{Eq:1:5} and
	\begin{equation*}
		\left|\left\{u(\cdot, t_o)\geq M\right\}\cap K_\varrho(x_o)\right|
		\geq
		\alpha \big|K_\varrho\big|.
	\end{equation*}
Then there exist constants $\delta$ and $\eta$ in $(0,1)$ depending only on the data and $\alpha$,
such that 
\begin{equation*}
	u\geq\eta M
	\quad
	\mbox{a.e.~in $K_{2\varrho}(x_o)\times\big(t_o+\delta(\frac{1}{2}\varrho)^p,t_o+\delta\varrho^{sp}\big],$}
\end{equation*}

\end{proposition}

\begin{remark}\label{Rem:1.1}\upshape
	By repeated applications of Proposition \ref{Prop:1:1},
	we could conclude that for an arbitrary $A>1$,
	there exists some $\bar\eta\in (0,1)$ depending on $\al$, the data and also on $A$, such that
	\begin{equation*}
	u\ge\bar\eta M
	\quad
	\mbox{a.e.~in $K_{2\varrho}(x_o)\times(t_o+\varrho^p,t_o+A\varrho^p]$,}
\end{equation*}
provided this cylinder is included in $E_T$. 
\end{remark}
\subsection{De Giorgi Lemma}
Here we prove a De Giorgi-type Lemma on cylinders of the form $Q_\rho(\theta)$. In the proof of \cref{Prop:1:1} $\theta \in (0,1)$ will be a universal constant depending only on the data.
\begin{lemma}\label{Lm:3:3}
Let $u$ be a globally nonnegative weak supersolution to ~\cref{eq:main-eq} in $E_T$ and
$(x_o,t_o) + Q_\varrho(\theta) =K_\varrho (x_o)\times (t_o-\theta\varrho^{sp}, t_o]\Subset E_T$ with $\theta \in (0,1)$. Suppose that $M>0$. Then there exists a constant $\nu\in(0,1)$ depending only 
the data and $\theta$, such that if
\begin{equation*}
	\Big|\Big\{
	 u \le M\Big\}\cap (x_o,t_o)+Q_{\varrho}(\theta)\Big|
	\le
	\nu|Q_{\varrho}(\theta)|,
\end{equation*}
then
\begin{equation*}
	u\ge\tfrac{1}2M
	\quad
	\mbox{a.e.~in $(x_o,t_o)+Q_{\frac{1}2\varrho}(\theta)$.}
\end{equation*}
Further the smallness parameter $\nu$ depends on the scaling parameter $\theta$ as
    \begin{equation}
        \nu = B_1^{-1}\theta^{\beta}
    \end{equation}
    for a constant $B_1 >1$ depending only on the data and an exponent $\beta>1$ depending only on $N$ and $p$.
\end{lemma}

\begin{proof}
Without loss of generality assume that $(x_o,t_o)=(0,0)$. For $n=0.1,\ldots$ set
\begin{align}\label{choices:k_n}
	\left\{
	\begin{array}{c}
	\displaystyle k_n=\frac{M}2+\frac{M}{2^{n+1}},\quad \tilde{k}_n=\frac{k_n+k_{n+1}}2,\\[5pt]
	\displaystyle \varrho_n=\frac{\varrho}2+\frac{\varrho}{2^{n+1}},
	\quad\tilde{\varrho}_n=\frac{\varrho_n+\varrho_{n+1}}2,\\[5pt]
	\displaystyle K_n=K_{\varrho_n},\quad \widetilde{K}_n=K_{\tilde{\varrho}_n},\\[5pt] 
	\displaystyle Q_n=Q_{\rho_n}(\theta),\quad
	\widetilde{Q}_n=Q_{\tilde\rho_n}(\theta).
	\end{array}
	\right.
\end{align}
and consider a cutoff function $0\le\zeta\le 1$ supported in $\widetilde{Q}_{n}$ and equal to identity on $Q_{n+1}$, such that
\begin{equation*}
	|\de\zeta|\leq\boldsymbol\gamma\frac{2^n}{\varrho}
	\quad\text{and}\quad 
	|\de_t\zeta|\leq\boldsymbol\gamma\frac{2^{pn}}{\theta\varrho^{sp}}.
\end{equation*}
We note that by Lemma~\ref{lem:g} we have
\begin{equation*}
	\mathfrak g_-(u,k)
	\le
	\boldsymbol\gamma \big(|u|+|k|\big)^{p-2}(u-k)_-^2
	\le
	\boldsymbol\gamma \big(|u|+|k|\big)^{p-1}(u-k)_-
\end{equation*}
and for $\tilde k<k$ there holds $(u-k)_-\ge (u-\tilde k)_-$. Then applying the energy estimates over $K_n$ and $Q_n$ to $ (u-k_n)_{-}$ we get

\begin{align*}
    \underset{-\theta\varrho_n^{sp} < t < 0}{\esssup}\int_{K_n}(|u|+|k_n|)^{p-2}w^2\zeta^p(x,t)\,dx  + \underset{Q_n}{\iint}\frac{|(u-k_n)_{-}(x,t)\zeta(x,t)-(u-k_n)_{-}\zeta(y,t)|^p}{|x-y|^{N+sp}}\,dx \,dy\,dt
    \\ 
    \leq \gamma\frac{2^{np}}{\varrho^{sp}}M^p|[u<k_n]\cap Q_n| + \gamma\frac{2^{np}}{\theta\varrho^{sp}}\iint_{Q_n}(|u|+|k_n|)^{p-1}(u-k_n)_{-}\,dx\,dt 
    \\
    + C\underset{\stackrel{-\theta\varrho_n^{sp} < t < 0;}{ x\in \text{supp }\zeta}}{\esssup}\int_{K_n^c}\frac{(u-k_n)_{-}^{p-1}(y,t)}{|x-y|^{N+sp}}\,dy\int_{-\theta\varrho_n^{sp}}^{0}\int_{K_n} (u-k_n)_{-}(x,t)\zeta(x,t)\,dx\,dt
\end{align*}
Next, we make the following observations 
\begin{itemize}
    \item $0 \le u\le k_n\le M$ on $\big\{u<k_n\big\}$,
    \item $|u|+|k_n|\leq 2 M$ on $\big\{u<k_n\big\}$,
    \item $|u|+|k_n|\geq k_n-u\ge k_n- \tilde{k}_{n}=2^{-(n+3)}M$ on the set $\{u<\tilde{k}_n\}$.
\end{itemize} 
and estimate the Tail as below noting that the global non-negativity of $u$ implies that $(u-k_n)_{-} \leq k_n \leq M$ (see ~\cref{sec:tail}).
\begin{align*}
\underset{\stackrel{-\theta\varrho_n^{sp} < t < 0;}{ x\in \text{supp }\zeta}}{\esssup}\int_{K_n^c}\frac{(u-k_n)_{-}^{p-1}(y,t)}{|x-y|^{N+sp}}\,dy 
\leq c2^{n(N+sp)}\frac{M^{p-1}}{\varrho^{sp}}
\end{align*}
for a universal constant $c>0$ to get

\begin{equation}\label{Eq:sample}
\begin{aligned}
	\frac{M^{p-2}}{2^{p(n+3)}} 
	\int_{\widetilde{K}_n} (u-\tilde{k}_n)_-^2\,dx
	+\underset{\widetilde{Q}_n}{\iint}\frac{|(u-\tilde{k}_n)_{-}(x,t)-(u-\tilde{k}_n)_{-}|^p}{|x-y|^{N+sp}}\,dx \,dy\,dt
	\leq
	\boldsymbol\gamma \frac{2^{n(N+2p)}}{\theta\varrho^{sp}}M^{p}|A_n|.
\end{aligned}
\end{equation}
To ease notation, we  $w = (u-\tilde{k}_n)_{-}$. Now setting $0\leq\phi\leq1$ to be a cutoff function which vanishes on the parabolic boundary $\widetilde{Q}_n$
and equals one on $Q_{n+1}$, an application of the H\"older inequality  and the Sobolev embedding (\cref{fracpoin}) gives that
gives that
\begin{align*}
	\frac{M}{2^{n+3}}
	|A_{n+1}|
	&\leq
	\iint_{\widetilde{Q}_n}\big(u-\tilde{k}_n\big)_-\phi\,dx\,dt\\
	&\leq
	\bigg[\iint_{\widetilde{Q}_n}\big[\big(u-\tilde{k}_n\big)_-\phi\big]^{p\frac{N+2s}{N}}
	\,dx\,dt\bigg]^{\frac{N}{p(N+2s)}}|A_n|^{1-\frac{N}{p(N+2s)}}\\
	&\leq\boldsymbol\gamma
	\left(\rho^{sp}\int_{-\theta\rho_n^{sp}}^{0}\int_{\widetilde{K}_n}\int_{\widetilde{K}_n}\frac{|w\phi(x,t)-w\phi(y,t)|^p}{|x-y|^{N+sp}}\,dx \,dy\,dt + \int_{-\theta\rho_n^{sp}}^{0}\int_{\widetilde{K}_n}[w\phi(x,t)]^p\,dx\,dt\right)^{\frac{N}{p(N+2s)}} 
    \\&\qquad\times \left(\underset{-\theta\rho_n^{sp} < t < 0}{\esssup}\int_{\widetilde{K}_n}[w\phi(x,t)]^2\,dx\right)^{\frac{s}{N+2s}}|A_n|^{1-\frac{N}{p(N+2s)}}\\
	&\leq
	\boldsymbol\gamma \theta^{-\frac{N+sp}{p(N+2s)}}
	\frac{b_0^n}{\varrho^\frac{s(N+sp)}{N+2s}}
	M |A_n|^{1+\frac{s}{N+2s}}.
\end{align*}
for a universal constant $b_0 > 1$. Setting
\[
\boldsymbol  Y_n=|A_n|/|Q_n|
\]
we get
\begin{equation*}
	\boldsymbol  Y_{n+1}
	\le
	\boldsymbol\gamma \boldsymbol b^n  \theta^{-\frac{N}{p(N+2s)}} Y_n^{1+\frac{1}{N+2s}},
\end{equation*}
for a constants $\boldsymbol\gamma > 0$ and $\boldsymbol b>1$ depending only on the data. The conclusion now follows from an application of the iteration lemma (\cref{geo_con}.) 
\end{proof}

\subsection{Expansion of positivity in time}\label{sec:exptime}

\begin{lemma}\label{Lm:3:1}
	Let $M>0$ and $\alpha \in(0,1)$. Then, there exist $\delta$ and $\epsilon$ in $(0,1)$,
	depending only on the data and $\alpha$, such that whenever $u$ is a globally non-negative weak supersolution to \cref{eq:main-eq} in $E_T$ satisfying
	\begin{equation*}
	\Big|\Big\{
		u(\cdot, t_o)\geq M
		\Big\}\cap K_{\varrho}(x_o)\Big|
		\geq\alpha \big|K_{\varrho}\big|,
	\end{equation*}
	then 
	\begin{equation}\label{Eq:3:1}
	\Big|\Big\{
	u(\cdot, t)\geq \epsilon M\Big\} \cap K_{\varrho}(x_o)\Big|
	\geq\frac{\alpha}2 |K_\varrho|
	\quad\mbox{ for all $t\in(t_o,t_o+\delta\varrho^{sp}]$.}
\end{equation}
Further the expansion parameter $\delta$ and the reduction parameter $\ep$ depends on the scaling parameter $\alpha$ as
    \begin{equation}
        \delta = B_2^{-1}\alpha^{p+1} \quad \text{and} \quad \ep = B_3^{-1}\alpha
    \end{equation}
    for a constants $B_2,B_3 >1$ depending only on the data. 
\end{lemma}
\begin{proof} 
Without loss of generality assume that $(x_o,t_o)=(0,0)$. For $k>0$ and $t>0$ set
\[
A_{k,\rho}(t) = [u(\cdot,t) < k] \cap K_{\rho}.
\]
By hypothesis, 
\[
|A_{M,\rho}(0)|\leq (1-\alpha)|K_{\rho}|.
\]
We consider the energy estimate for $(u-k)_{-}$ with $k=M$ over the cylinder $K_{\rho}\times(0,\delta\rho^{sp}]$ where $\delta>0$ will be chosen later. Note that $(u-k)_{-} \leq M$ globally because of the non-negativity of $u$. For $\sigma \in (0,1/8]$ to be chosen later, we take a cutoff function $\zeta = \zeta(x)$, nonnegative such that it is supported in $K_{(1-\sigma/2)\rho}$, $\zeta = 1$ on $K_{(1-\sigma)\rho}$ and $|\de \zeta| \leq 2(\sigma\rho)^{-1}$ to get
\begin{align*}
    \int_{K_\varrho\times\{t\}}\int_{u}^k |s|^{p-2}(s-k)_-\,ds \zeta^p&\,dx\\ 
    \leq \frac{\gamma M^p}{(\sigma \rho)^{sp}}\delta\rho^{sp}|K_{\rho}|
    &+ \int_{K_\varrho\times\{0\}}\int_{u}^k     |s|^{p-2}(s-k)_-\,ds\zeta^p\,dx
    \\
    &+ C\underset{\stackrel{t \in (0,\delta\rho^{sp}]}{ x\in \text{supp }\zeta}}{\esssup}\int_{K_{\rho}^c}\frac{(u-k)_{-}^{p-1}(y,t)}{|x-y|^{N+sp}}\,dy\int_{0}^{\delta\rho^{sp}}\int_{K_{\rho}} (u-k)_{-}(x,t)\zeta(x,t)\,dx\,dt
\end{align*}
Using $u\geq 0$ we get
\begin{align*}
	\int_{K_\varrho\times\{0\}}\int_{u}^k |s|^{p-2}(s-k)_-\,ds\zeta^p\,dx
	\leq 
	|A_{M,\rho}(0)|\int_{0}^{k} |s|^{p-2}(s-k)_-\,ds
	\leq
	(1-\alpha)|K_\varrho|\int_{0}^{k} |s|^{p-2}(s-k)_-\,ds.
\end{align*}
For the Tail term, we follow what was described in \cref{sec:tail} to get
\begin{align*}
\underset{\stackrel{t \in (0,\delta\rho^{sp}]}{ x\in \text{supp }\zeta}}{\esssup}\int_{K_{\rho}^c}\frac{(u-k)_{-}^{p-1}(y,t)}{|x-y|^{N+sp}}\,dy
\leq c\sigma^{-(N+sp)}\frac{M^{p-1}}{\varrho^{sp}}
\end{align*}
for a universal constant $c>0$.
We estimate from below
\begin{align*}
	\int_{K_\varrho\times\{t\}}\int_{u}^k |s|^{p-2}(s-k)_-\,ds \zeta^p\,dx
	\geq 
	\big|A_{k_\epsilon,(1-\sigma)\varrho}(t)\big| \int_{k_\epsilon}^k |s|^{p-2}(s-k)_-\,ds
\end{align*}
where 
\begin{equation*}
	k_\epsilon=\epsilon M
\end{equation*}
and $\epsilon\in(0,\frac{1}{2})$ will be chosen later. We now use
\[
\frac{1}{2}M\leq(1-\epsilon)M=k-k_\epsilon\leq |k_\epsilon|+|k|\leq 2M
\]
to estimate from below again
\begin{align}\label{prop_1}
	\int_{k_\epsilon}^k|s|^{p-2}(s-k)_-\,ds
	=
	\tfrac{1}{p-1}\, \mathfrak g_-(k_\epsilon,k)
	\geq 
	\tfrac{1}{\boldsymbol\gamma (p)}
	\big(|k_\epsilon|+|k|\big)^{p-2} (k-k_\epsilon)^2
	\geq
	\tfrac{1}{\boldsymbol\gamma (p)} M^p.
\end{align}
Next, we note that
\begin{align*}
	\big|A_{k_\epsilon,\varrho}(t)\big|
	&=
	\big|A_{k_\epsilon,(1-\sigma)\varrho}(t)\cup (A_{k_\epsilon,\varrho}(t)\setminus A_{k_\epsilon,(1-\sigma)\varrho}(t))\big|\\
	&\leq 
	\big|A_{k_\epsilon,(1-\sigma)\varrho}(t)\big|+|K_\varrho\setminus K_{(1-\sigma)\varrho}|\\
	&\leq
	\big |A_{k_\epsilon,(1-\sigma)\varrho}(t)\big|+N\sigma |K_\varrho|.
\end{align*}
Putting together the above estimates gives
\begin{align*}
	|A_{k_\epsilon,\varrho}(t)|
	\leq 
	\frac{\int_{0}^k |s|^{p-2}(s-k)_-\,ds }{\int_{k_\epsilon}^k |s|^{p-2}(s-k)_-\,ds }
	(1-\alpha)|K_\varrho|
	+
	\frac{\boldsymbol\gamma \delta}{\sigma^p}|K_\varrho| +N\sigma|K_\varrho|
\end{align*}
for a universal constant $\boldsymbol\gamma > 0$. We write
\[
\frac{\int_{0}^k |s|^{p-2}(s-k)_{-} \,ds }{\int_{k_\epsilon}^k |s|^{p-2}(s-k)_{-} \,ds } = 1 + \frac{\int_{0}^{k_\epsilon} |s|^{p-2}(s-k)_{-} \,ds}{\int_{k_\epsilon}^k  |s|^{p-2}(s-k)_{-} \,ds} = 1 + I_{\epsilon}.
\]
Using $|k_\epsilon|\leq M$ and applying \cref{lem:fusco}  we get
\[
	\int_{0}^{k_\epsilon} |\tau|^{p-2}(\tau-k)_-\,d\tau
	\leq 
	M\int_{0}^{k_\epsilon}|\tau|^{p-2}\,d\tau
	=
	M |s|^{p-2} s\Big|_{0}^{k_\epsilon}
	\leq
	\boldsymbol\gamma (p)M^p\epsilon.
\]
And so using \cref{prop_1} we obtain
\begin{align*}
	I_\epsilon
	\leq
	\boldsymbol\gamma (p) \epsilon.
\end{align*}
We choose $\epsilon\in (0,1)$ small enough such that
\begin{equation*}
	(1-\alpha)(1+\boldsymbol \gamma\epsilon)\leq 1-\tfrac{3}{4}\alpha.
\end{equation*}
and define $\sigma:=\frac{\alpha}{8N}$.
Finally, we choose $\delta\in (0,1)$ small enough so that
$$
	\frac{\boldsymbol \gamma\delta}{\sigma^p}\leq\frac{\alpha}{8}.
$$
The conclusion follows. 
\end{proof}

\subsection{A Shrinking Lemma}

\begin{lemma}\label{Lm:3:2}
Let $M>0$ and $\alpha \in(0,1)$. Suppose that $u$ is a globally non-negative weak supersolution to \cref{eq:main-eq} in $E_T$ satisfying
	\begin{equation*}
	\Big|\Big\{
		u(\cdot, t_o)\geq M
		\Big\}\cap K_{\varrho}(x_o)\Big|
		\geq\alpha \big|K_{\varrho}\big|,
	\end{equation*}
Let $Q=K_{\varrho}(x_o)\times\left(t_o,t_o+\delta\rho^{sp}\right]$ and let $\widehat Q=K_{4\varrho}(x_o)\times\left(t_o,t_o+\delta\rho^{sp}\right]\Subset E_T$. 
Then there exists $\boldsymbol \gamma>0$ depending only on the data such that for any positive integer $j$  we have 
\begin{equation*}
	\bigg|\bigg\{
	u\le\frac{\epsilon M}{2^{j+1}}\bigg\}\cap \widehat Q\bigg|
	\leq \frac{1}{\al}\left(\frac{\gamma}{2^{j}-1}\right)^{p-1}|\widehat Q|,\quad
\end{equation*}
\end{lemma}
\begin{proof}
Without loss of generality assume that $(x_o,t_o)=(0,0)$. We write the energy estimates over the cylinder 
\[
K_{8\rho}\times (0,\delta\rho^{sp}]
\]
for the functions $(u-k_j)_-$ with the choice of levels 
\[
k_j =  \frac{\epsilon M}{2^j} \qquad j \geq 1
\] 
We choose a test function $\zeta = \zeta(x)$ such that $\zeta\equiv 1$ on $K_{4\varrho}$, it is supported in $K_{6\varrho}$ and
\[
|\de \zeta|\leq \frac{1}{2\rho}
\] 
Using $(u-k_j)_{-} \leq \epsilon M2^{-j}$ globally, by our choice of the test function we get
\begin{align*}
\int_{0}^{\delta\varrho^{sp}}\int\limits_{K_{4\varrho}}|(u-k_j)_-(x,t)|\int\limits_{K_{4\varrho}}\frac{|(u-k_j)_+(y,t)|}{|x-y|^{N+sp}}\,dx\,dy\,dt
\\
\leq C\frac{1}{\varrho^{sp}}\left(\frac{\epsilon M}{2^j}\right)^{p}|\widehat Q|
     +C\frac{\epsilon M}{2^j}|\widehat Q|\underset{\stackrel{t \in (0,\delta\rho^{sp})}{ x\in \text{supp }\zeta}}{\esssup}\int_{K_{8\rho}^c}\frac{(u-k_j)_{-}(y,t)}{|x-y|^{N+sp}}\,dy
     \\
     + C\int_{K_{8\varrho}\times\{0\}} \mathfrak g_- (u,k_j) \,dx
\end{align*}
where $C>0$ is a universal constant. We recall the section on Tail estimates (\cref{sec:tail}) to get
\begin{align*}
\underset{\stackrel{t \in (0,\delta\rho^{sp}]}{ x\in \text{supp }\zeta}}{\esssup}\int_{K_{8\rho}^c}\frac{(u-k_j)_{-}^{p-1}(y,t)}{|x-y|^{N+sp}}\,dy 
\leq c\frac{1}{\varrho^{sp}}\left(\frac{\epsilon M}{2^j}\right)^{p-1}
\end{align*}
for a universal constant $c>0$. Next, we estimate the time term as follows. From \cref{lem:g} we get 
\begin{equation*}
	\mathfrak g_- (u,k_j)
	\le
	\boldsymbol\gamma \big(|u|+|k_j|\big)^{p-2}(u-k_j)_-^2.
\end{equation*}
For $p\geq 2$, we use $(u-k_j)_-\le |u|+|k_j|$, $u\geq 0$ to get
\begin{equation*}
	\mathfrak g_- (u,k_j)
	\leq
	\boldsymbol\gamma \big(|u|+|k_j|\big)^p \chi_{\{u\leq k_j\}}
	\leq
	\boldsymbol\gamma\left(\frac{\epsilon M}{2^j}\right)^{p}.
\end{equation*}
For $1<p<2$, we use $(u-k_j)_-\leq |u|+|k_j|$ and $u\ge 0$ to get
\begin{align*}
	\mathfrak g_- (u,k_j)
	\leq
	\boldsymbol\gamma (u-k_j)_-^p
	\leq
	\boldsymbol\gamma\left(\frac{\epsilon M}{2^j}\right)^{p},
\end{align*}
for a constant $\boldsymbol\gamma$ depending only on $p$. So for all $1<p<\infty$ we get
\begin{equation*}
	\int_{K_{8\varrho}\times\{0\}}\mathfrak g_- (u,k_j)
	\leq
	 \frac{\boldsymbol \gamma}{\delta\varrho^{sp}}
	 \left(\frac{\epsilon M}{2^j}\right)^{p} |\widehat Q|
\end{equation*}
Putting the above estimates together and recalling that $\delta \in (0,1)$ we get
\[
\int_{0}^{\delta\varrho^{sp}}\int\limits_{K_{4\varrho}}|(u-k_j)_-(x,t)|\int\limits_{K_{4\varrho}}\frac{|(u-k_j)_+(y,t)|}{|x-y|^{N+sp}}\,dx\,dy\,dt \leq \frac{\boldsymbol \gamma}{\delta\varrho^{sp}}
	 \left(\frac{\epsilon M}{2^j}\right)^{p} |\widehat Q|
\]
We now invoke Lemma \ref{lem:shrinking} with Step 1 to conclude that
\[
|\{u < \epsilon M2^{-(j+1)}\}\cap\widehat Q)| \leq \frac{1}{\al}\left(\frac{C}{2^{j}-1}\right)^{p-1}|\widehat Q|
\]
for a constant $C>0$ depending only on the data.
\end{proof}

\subsection{Proof of expansion of positivity}
Without loss of generality assume that $(x_0,t_0) = (0,0)$. By $\delta,\epsilon\in(0,1)$ and $\boldsymbol \gamma>0$ we denote the corresponding constants from \cref{Lm:3:1} and \cref{Lm:3:2} depending on the data and $\alpha$ and by $\nu\in(0,1)$ we denote the constant from \cref{Lm:3:3} applied with $\theta=\delta$. Then, $\nu$ depends on the data and $\alpha$. 
Next, we choose an integer $j \geq 1$ in such a way that 
\[
\left(\frac{\gamma}{2^{j}-1}\right)^{p-1} \leq \nu
\]
Then, $j$ depends only on the data and $\alpha$. Applying \cref{Lm:3:1} and \cref{Lm:3:2} we get
\begin{equation*}
	\bigg|\bigg\{
	u\leq \frac{\epsilon M}{2^{j+1}}\bigg\}\cap \widehat Q\bigg|
	\leq
	\nu|\widehat Q|,\quad
\end{equation*}
where $\widehat Q=K_{4\varrho}\times\left(0,\delta\varrho^{sp}\right]$
(note that $0 < \epsilon < 1$.) Finally, we apply \cref{Lm:3:3} with $M$ replaced by $\frac{\epsilon M}{2^{j+1}}$ to get 
\begin{equation*}
	u\geq \frac{\epsilon M}{2^{j+2}}
	\quad
	\mbox{a.e.~in $K_{2\rho}\times \left(\delta\left(\frac{1}{2}\varrho\right)^{sp
	},\delta\varrho^{sp} \right]$.}
\end{equation*}
We set
$$\eta=\frac{\epsilon}{2^{j+2}}.$$ The conclusion follows. 
\begin{remark}\label{rem:polydep}
    Note the tracing the dependency on $\al$ we get that the dependence of $\eta$ on $\alpha$ is polynomial. This is \textit{not} true in the local case and is a feature of nonlocal equations. On the other hand, note that $\delta$ also has a polynomial dependence on $\al$ just like in the local setup. 
\end{remark}
\section{Weak Harnack Estimate}

\subsection{Power Dependence Free Positivity Expansion}
We now remove the power like dependency of $\delta$ on $\al$ in the expansion of positivity using the clustering lemma. 
\begin{proposition}\label{Prop:nopower}
	Let $u$ be a globally nonnegative weak super-solution to \cref{eq:main-eq} in $E_T$.
	Suppose for some $(x_o,t_o)\in E_T$, $M>0$, $\alpha \in(0,1)$ and $\varrho>0$ we have \cref{Eq:1:5} and
	\begin{equation*}
		\left|\left\{u(\cdot, t_o)\geq M\right\}\cap K_\varrho(x_o)\right|
		\geq
		\alpha \big|K_\varrho\big|.
	\end{equation*}
Then there exist constants $\delta_0$ and $\eta_0$ in $(0,1)$ and $d>1$ depending only on the data 
such that 
\begin{equation*}
	u\geq\eta_0\alpha^{d} M
	\quad
	\mbox{a.e.~in $K_{2\varrho}(x_o)\times\big(t_o+\delta_0(\frac{1}{2}\varrho)^p,t_o+\delta_0\varrho^{sp}\big],$}
\end{equation*}

\end{proposition}
\begin{proof}
  Without loss of generality assume that $(x_0,t_0) = (0,0)$. By \cref{Lm:3:1} there exist constants $B_2,B_2 > 1$ depending only on data such that 
  \[
    \delta = B_2^{-1}\alpha^{p+1} \quad \text{and} \quad \ep = B_3^{-1}\alpha
  \]   
  and 
  \begin{equation}\label{eq:measinfo}
    |\{u(\cdot,t)>\ep M\} \cap K_{\rho}| \geq \frac{\al}{2}|K_{\rho}|\quad \forall \quad 0<t<\delta\rho^{sp}.     
  \end{equation}

  Let $Q = K_{\rho} \times (\delta(\frac{1}{2}\rho)^{sp},\delta\rho^{sp}]$ and $Q' = K_{2\rho}\times(0,\delta\rho^{sp}]$. Let  $0 \leq \xi \leq 1$ be a cutoff function which is one on $Q$, vanishes on the parabolic boundary $\de_pQ'$ such that
  \[
    0\leq \xi_t \leq \frac{4^p}{\delta\rho^{sp}} \quad \text{ and } |\de\xi| \leq \frac{4}{\rho}
  \]
  Recalling that $u$ is globally non-negative, we apply the energy estimates for $(u-M)_{-}$ over the cylinders $Q \subset Q'$ with the cutoff $\xi$ to get
  \[
  \int_{\delta(\frac{1}{2}\rho)^{sp}}^{\delta\rho^{sp}}\int_{K_{\rho}}\int_{K_{\rho}}\frac{|(u-M)_{-}(x,t)-(u-M)_{-}(y,t)|^p}{|x-y|^{N+sp}} \,dy\,dx\,dt \leq \gamma\frac{M^p}{\delta\rho^{sp}}|Q|.
  \]
  We change variables
  \[
    X = \frac{x}{\rho} \quad  Y = \frac{Y}{\rho} \quad T = \frac{t}{\delta\rho^{sp}} \quad w = \frac{(u-M)_{-}}{M}
  \]
  and recall the dependence of $\delta$ on $\al$ to get
  \[
  \int_{\frac{1}{2^{sp}}}^{1}\int_{K_1}\int_{K_1}\frac{|w(X,T)-w(Y,T)|^p}{|X-Y|^{N+sp}} \,dY\,dX\,dT \leq \frac{\gamma}{\al^{p+1}}.
  \]
  where $\gamma$ is a constant that depends only on the data. Setting $v = (1-w)/\ep$, we get
  \[
  \left(\int_{\frac{1}{2^{sp}}}^{1}\int_{K_1}\int_{K_1}\frac{|v(X,T)-v(Y,T)|^p}{|X-Y|^{N+sp}} \,dY\,dX\,dT\right)^{\frac{1}{p}} \leq \frac{\gamma}{\ep\al^{\frac{p+1}{p}}}.
  \]
  whereas \cref{eq:measinfo} becomes
  \[
  |\{v(\cdot,T)>1\}\cap K_1| \geq \frac{\al}{2}|K_1| \quad \forall \quad \frac{1}{2^{sp}} < T < 1. 
  \]
  So we can find a $T_1 \in (2^{-sp},1]$ such that
  \[
   \left(\int_{K_1}\int_{K_1}\frac{|v(x,T_1)-v(y,T_1)|}{|x-y|^{N+s}} \,dy\,dx\right)^{\frac{1}{p}} \leq \frac{\gamma}{\ep\al^{\frac{p+1}{p}}}.
  \]
  and
  \[
  |\{v(\cdot,T_1)>1\}\cap K_1| \geq \frac{\al}{2}|K_1|.
  \]
  Recalling the clustering lemma along with the dependence of the reducing parameter and $\ep$ on $\al$ shows that there exist $X_0 \in K_1$ and $\sigma = C^{-1}\al^{3p+2}$ for some constant $C>1$ depending only on data such that
  \[
   \{v(\cdot,T_1) > \frac{1}{2}\}\cap K_{\sigma}(X_0)| \geq \frac{1}{2}|K_{\sigma}|.
  \]
  In the original coordinates we get
  \[
  \{u(\cdot,t_1) > \frac{1}{2}\ep M\}\cap K_{\sigma\rho}(x_0)| \geq \frac{1}{2}|K_{\sigma\rho}|.
  \]
  for some $x_0 \in K_{\rho}$ and $\delta2^{-sp}\rho^{sp}<t_1<\delta\rho^{sp}$. We now apply the expansion of positivity \cref{Prop:1:1} repeatedly with $\al = 1/2$ to obtain $\eta'$ and $\delta'$ in $(0,1)$ depending only on the data such that for $n \geq 1$
  \[
  u(\cdot,t) \geq \frac{1}{2}\ep\eta'^nM \quad \text{a.e. in } K_{2^n\sigma\rho}(x_0)
  \]
  for all
  \[
  t_{n-1} + \delta'\left(2^{n-1}\sigma\rho\right)^{sp} < t < t_n \coloneqq t_{n-1}+ \delta'\left(2^{n}\sigma\rho\right)^{sp}.
  \]
  To get the expansion of positivity to the $K_{2\rho}$ we take
  \[
  2^n\sigma = 2 
  \]
  which implies
  \[
  n = 1 + \ln \sigma^{-\frac{1}{\ln 2}} = 1 + \log_{\frac{1}{\eta'}}\sigma^{\frac{\ln \eta'}{\ln 2}}.
  \]
  Thus recalling the polynomial dependence of $\ep$ and $\sigma$ on the parameter $\al$  we obtain the existence of $\eta_0 \in (0,1)$ and $d>1$ depending only on the data such the the conclusion holds. 
\end{proof}
\subsection{Proof of Weak Harnack}
Without loss of generality assume that $(x_0,t_0) = (0,0)$. We fix the constants $d>1$ and $\eta_0, \delta_0 \in (0,1)$ from \cref{Prop:nopower}. Recall that they depend only on the data. Let 
\[
\lambda_0 \quad = \underset{K_{2\rho}\times(\delta_0(\rho/2)^{sp},\delta_0\rho^{sp}]}{\essinf} u .
\]
For $0<q<1/d$ we compute
\[
\int_{K_{\rho}}u^q(\cdot,0) \,dx \leq \lambda_0^q|K_{\rho}| + q\int_{\lambda_0}^{\infty}M^{q-1}|\{u(\cdot,0)>M\}\cap K_{\rho}| \,dM. 
\]
From \cref{Prop:nopower} we get
\[
\lambda_0 \geq \eta_0M \left(\frac{|\{u(\cdot,0)>M\}\cap K_{\rho}|}{|K_{\rho}|}\right)^d.
\]
It follows that
\[
q\int_{\lambda_0}^{\infty}M^{q-1}|\{u(\cdot,0)>M\}\cap K_{\rho}| \,dM \leq \frac{q\lambda_0^{\frac{1}{d}}}{\eta_0^{\frac{1}{d}}}|K_{\rho}|\int_{\lambda_0}^{\infty}M^{q-\frac{1}{d}-1}\,dM.
\]
Thus
\[
\fint_{K_{\rho}}u^q(\cdot,0) \,dx \leq \lambda_0^q\left(1+\frac{q}{\eta_0^{\frac{1}{d}}(\frac{1}{d}-q)}\right).
\]
\section*{Acknowledgments} 
The author would like to thank Karthik Adimurthi, Agnid Banerjee and Vivek Tewary for helpful comments and suggestions. The author was supported by the Department of Atomic Energy,  Government of India, under	project no.  12-R\&D-TFR-5.01-0520. 

\bibliography{main}

\begin{thebibliography}{47}
\providecommand{\natexlab}[1]{#1}
\providecommand{\url}[1]{\texttt{#1}}
\expandafter\ifx\csname urlstyle\endcsname\relax
  \providecommand{\doi}[1]{doi: #1}\else
  \providecommand{\doi}{doi: \begingroup \urlstyle{rm}\Url}\fi

\bibitem[Acerbi and Fusco(1989)]{ACERBI1989115}
E~Acerbi and N~Fusco.
\newblock Regularity for minimizers of non-quadratic functionals: The case 1 <
  p < 2.
\newblock \emph{Journal of Mathematical Analysis and Applications},
  140\penalty0 (1):\penalty0 115--135, 1989.
\newblock ISSN 0022-247X.
\newblock \doi{https://doi.org/10.1016/0022-247X(89)90098-X}.
\newblock URL
  \url{https://www.sciencedirect.com/science/article/pii/0022247X8990098X}.

\bibitem[Adimurthi et~al.(2022)Adimurthi, Prasad, and Tewary]{aptHolder}
Karthik Adimurthi, Harsh Prasad, and Vivek Tewary.
\newblock Local {H\"{o}lder} regularity for nonlocal parabolic $p$-{Laplace}
  equations, 2022.
\newblock URL \url{https://arxiv.org/abs/2205.09695}.

\bibitem[Adimurthi et~al.(2023)Adimurthi, Prasad, and
  Tewary]{adimurthi2023holder}
Karthik Adimurthi, Harsh Prasad, and Vivek Tewary.
\newblock H{\"o}lder regularity for fractional p-{Laplace} equations.
\newblock \emph{Proceedings-Mathematical Sciences}, 133\penalty0 (1):\penalty0
  14, 2023.
\newblock \doi{10.1007/s12044-023-00734-6}.
\newblock URL
  \url{https://link.springer.com/article/10.1007/s12044-023-00734-6}.

\bibitem[Alonso et~al.(2008)Alonso, Santillana, and Dawson]{ALONSO2008}
R.~Alonso, M.~Santillana, and C.~Dawson.
\newblock On the diffusive wave approximation of the shallow water equations.
\newblock \emph{European Journal of Applied Mathematics}, 19\penalty0
  (5):\penalty0 575--606, October 2008.
\newblock \doi{10.1017/s0956792508007675}.
\newblock URL \url{https://doi.org/10.1017/s0956792508007675}.

\bibitem[Banerjee et~al.(2021)Banerjee, Garain, and Kinnunen]{Banerjee2021}
Agnid Banerjee, Prashanta Garain, and Juha Kinnunen.
\newblock Some local properties of subsolution and supersolutions for a doubly
  nonlinear nonlocal p-{Laplace} equation.
\newblock \emph{Annali di Matematica Pura ed Applicata (1923 -)}, 201\penalty0
  (4):\penalty0 1717--1751, November 2021.
\newblock \doi{10.1007/s10231-021-01177-4}.
\newblock URL \url{https://doi.org/10.1007/s10231-021-01177-4}.

\bibitem[Banerjee et~al.(2022)Banerjee, Garain, and Kinnunen]{Banerjee2022}
Agnid Banerjee, Prashanta Garain, and Juha Kinnunen.
\newblock Lower semicontinuity and pointwise behavior of supersolutions for
  some doubly nonlinear nonlocal parabolic p-{Laplace} equations.
\newblock \emph{Communications in Contemporary Mathematics}, June 2022.
\newblock \doi{10.1142/s0219199722500328}.
\newblock URL \url{https://doi.org/10.1142/s0219199722500328}.

\bibitem[B\"{o}gelein et~al.(2021)B\"{o}gelein, Duzaar, and Liao]{Bgelein2021}
Verena B\"{o}gelein, Frank Duzaar, and Naian Liao.
\newblock On the {H\"{o}lder} regularity of signed solutions to a doubly
  nonlinear equation.
\newblock \emph{Journal of Functional Analysis}, 281\penalty0 (9):\penalty0
  109173, November 2021.
\newblock \doi{10.1016/j.jfa.2021.109173}.
\newblock URL \url{https://doi.org/10.1016/j.jfa.2021.109173}.

\bibitem[B\"{o}gelein et~al.(2022)B\"{o}gelein, Duzaar, Liao, and
  Sch\"{a}tzler]{liaopart2}
Verena B\"{o}gelein, Frank Duzaar, Naian Liao, and Leah Sch\"{a}tzler.
\newblock On the {H\"{o}lder} regularity of signed solutions to a doubly
  nonlinear equation. part {II}.
\newblock \emph{Revista Matem{\'{a}}tica Iberoamericana}, May 2022.
\newblock \doi{10.4171/rmi/1342}.
\newblock URL \url{https://doi.org/10.4171/rmi/1342}.

\bibitem[Caffarelli(2012)]{Caffarelli2012}
Luis Caffarelli.
\newblock Non-local diffusions, drifts and games.
\newblock In \emph{Nonlinear Partial Differential Equations}, pages 37--52.
  Springer Berlin Heidelberg, 2012.
\newblock \doi{10.1007/978-3-642-25361-4\_3}.
\newblock URL \url{https://doi.org/10.1007/978-3-642-25361-4_3}.

\bibitem[Caffarelli et~al.(2011)Caffarelli, Chan, and
  Vasseur]{caffarelli2011regularity}
Luis Caffarelli, Chi~Hin Chan, and Alexis Vasseur.
\newblock Regularity theory for parabolic nonlinear integral operators.
\newblock \emph{Journal of the American Mathematical Society}, 24\penalty0
  (3):\penalty0 849--869, 2011.

\bibitem[Castro et~al.(2014)Castro, Kuusi, and Palatucci]{DiCastro2014}
Agnese~Di Castro, Tuomo Kuusi, and Giampiero Palatucci.
\newblock Nonlocal {Harnack} inequalities.
\newblock \emph{Journal of Functional Analysis}, 267\penalty0 (6):\penalty0
  1807--1836, September 2014.
\newblock \doi{10.1016/j.jfa.2014.05.023}.
\newblock URL \url{https://doi.org/10.1016/j.jfa.2014.05.023}.

\bibitem[Castro et~al.(2016)Castro, Kuusi, and Palatucci]{DiCastro2016}
Agnese~Di Castro, Tuomo Kuusi, and Giampiero Palatucci.
\newblock Local behavior of fractional p-minimizers.
\newblock \emph{Annales de l{\textquotesingle}Institut Henri Poincar{\'{e}} C,
  Analyse non lin{\'{e}}aire}, 33\penalty0 (5):\penalty0 1279--1299, October
  2016.
\newblock \doi{10.1016/j.anihpc.2015.04.003}.
\newblock URL \url{https://doi.org/10.1016/j.anihpc.2015.04.003}.

\bibitem[Cozzi(2017)]{Cozzi2017}
Matteo Cozzi.
\newblock Regularity results and {Harnack} inequalities for minimizers and
  solutions of nonlocal problems: A unified approach via fractional {De}
  {Giorgi} classes.
\newblock \emph{Journal of Functional Analysis}, 272\penalty0 (11):\penalty0
  4762--4837, June 2017.
\newblock \doi{10.1016/j.jfa.2017.02.016}.
\newblock URL \url{https://doi.org/10.1016/j.jfa.2017.02.016}.

\bibitem[DiBenedetto and Vespri(1995)]{clusterarma}
E.~DiBenedetto and V.~Vespri.
\newblock On the singular equation {$\beta(u)_t=\Delta u$}.
\newblock \emph{Archive for Rational Mechanics and Analysis}, 132\penalty0
  (3):\penalty0 247--309, 1995.
\newblock \doi{10.1007/bf00382749}.
\newblock URL \url{https://doi.org/10.1007/bf00382749}.

\bibitem[DiBenedetto(1993)]{DiBenedetto1993}
Emmanuele DiBenedetto.
\newblock \emph{Degenerate Parabolic Equations}.
\newblock Springer New York, 1993.
\newblock \doi{10.1007/978-1-4612-0895-2}.
\newblock URL \url{https://doi.org/10.1007/978-1-4612-0895-2}.

\bibitem[DiBenedetto et~al.(2006)DiBenedetto, Gianazza, and
  Vespri]{clusteringl1}
Emmanuele DiBenedetto, Ugo Gianazza, and Vincenzo Vespri.
\newblock Local clustering of the non-zero set of functions in {$W^{1,1}(E)$},
  journal = {Rendiconti Lincei - Matematica e Applicazioni}.
\newblock pages 223--225, 2006.
\newblock \doi{10.4171/rlm/465}.
\newblock URL \url{https://doi.org/10.4171/rlm/465}.

\bibitem[DiBenedetto et~al.(2008)DiBenedetto, Gianazza, and
  Vespri]{actaharnack}
Emmanuele DiBenedetto, Ugo Gianazza, and Vincenzo Vespri.
\newblock Harnack estimates for quasi-linear degenerate parabolic differential
  equations.
\newblock \emph{Acta Mathematica}, 200\penalty0 (2):\penalty0 181--209, 2008.
\newblock \doi{10.1007/s11511-008-0026-3}.
\newblock URL \url{https://doi.org/10.1007/s11511-008-0026-3}.

\bibitem[Dibenedetto et~al.(2008)Dibenedetto, Gianazza, and
  Vespri]{dukeharnack}
Emmanuele Dibenedetto, Ugo Gianazza, and Vincenzo Vespri.
\newblock Subpotential lower bounds for nonnegative solutions to certain
  quasi-linear degenerate parabolic equations.
\newblock \emph{Duke Mathematical Journal}, 143\penalty0 (1), May 2008.
\newblock \doi{10.1215/00127094-2008-013}.
\newblock URL \url{https://doi.org/10.1215/00127094-2008-013}.

\bibitem[DiBenedetto et~al.(2010)DiBenedetto, Gianazza, and
  Vespri]{anallisnsharnack}
Emmanuele DiBenedetto, Ugo Gianazza, and Vincenzo Vespri.
\newblock Forward, backward and elliptic harnack inequalities for non-negative
  solutions to certain singular parabolic partial differential equations.
\newblock \emph{Annali della Scuola Normale Superiore di Pisa-Classe di
  Scienze}, 9\penalty0 (2):\penalty0 385--422, 2010.

\bibitem[DiBenedetto et~al.(2012)DiBenedetto, Gianazza, and
  Vespri]{DiBenedetto2012}
Emmanuele DiBenedetto, Ugo Gianazza, and Vincenzo Vespri.
\newblock \emph{Harnack{\textquotesingle}s Inequality for Degenerate and
  Singular Parabolic Equations}.
\newblock Springer New York, 2012.
\newblock \doi{10.1007/978-1-4614-1584-8}.
\newblock URL \url{https://doi.org/10.1007/978-1-4614-1584-8}.

\bibitem[Ding et~al.(2021)Ding, Zhang, and Zhou]{Ding2021}
Mengyao Ding, Chao Zhang, and Shulin Zhou.
\newblock Local boundedness and {H\"{o}lder} continuity for the parabolic
  fractional p-{Laplace} equations.
\newblock \emph{Calculus of Variations and Partial Differential Equations},
  60\penalty0 (1), January 2021.
\newblock \doi{10.1007/s00526-020-01870-x}.
\newblock URL \url{https://doi.org/10.1007/s00526-020-01870-x}.

\bibitem[Dyda and Kassmann(2020)]{dyda2020regularity}
Bart{\l}omiej Dyda and Moritz Kassmann.
\newblock Regularity estimates for elliptic nonlocal operators.
\newblock \emph{Analysis \& PDE}, 13\penalty0 (2):\penalty0 317--370, 2020.

\bibitem[Düzgün et~al.(2023)Düzgün, Iannizzotto, and
  Vespri]{düzgün2023clustering}
Fatma~Gamza Düzgün, Antonio Iannizzotto, and Vincenzo Vespri.
\newblock A clustering theorem in fractional {Sobolev} spaces.
\newblock \emph{arXiv:2305.19965}, 2023.

\bibitem[Egorov(1911)]{egorov1911fonctions}
D~Egorov.
\newblock Sur les suites des fonctions meaurables.
\newblock \emph{Comptes Rendus de Acad. des Sc. de Paris}, 152:\penalty0
  244--246, 1911.

\bibitem[Felsinger and Kassmann(2013)]{Felsinger2013}
Matthieu Felsinger and Moritz Kassmann.
\newblock Local regularity for parabolic nonlocal operators.
\newblock \emph{Communications in Partial Differential Equations}, 38\penalty0
  (9):\penalty0 1539--1573, September 2013.
\newblock \doi{10.1080/03605302.2013.808211}.
\newblock URL \url{https://doi.org/10.1080/03605302.2013.808211}.

\bibitem[Giacomin et~al.(1998)Giacomin, Lebowitz, and
  Presutti]{giacomin1998deterministic}
Giambattista Giacomin, Joel~L Lebowitz, and Errico Presutti.
\newblock Deterministic and stochastic hydrodynamic equations arising from
  simple microscopic model systems.
\newblock \emph{Mathematical Surveys and Monographs}, 64:\penalty0 107--152,
  1998.

\bibitem[Gianazza and Vespri(2006)]{gianazza2006harnack}
Ugo Gianazza and Vincenzo Vespri.
\newblock A {Harnack} inequality for solutions of doubly nonlinear parabolic
  equations.
\newblock \emph{J. Appl. Funct. Anal}, 1\penalty0 (3):\penalty0 271--284, 2006.

\bibitem[Giaquinta and Modica(1986)]{Giaquinta1986}
Mariano Giaquinta and Giuseppe Modica.
\newblock Remarks on the regularity of the minimizers of certain degenerate
  functionals.
\newblock \emph{Manuscripta Mathematica}, 57\penalty0 (1):\penalty0 55--99,
  March 1986.
\newblock \doi{10.1007/bf01172492}.
\newblock URL \url{https://doi.org/10.1007/bf01172492}.

\bibitem[Ivanov(1992)]{ivanov1992quasilinear}
Aleksandr~Vasil'evich Ivanov.
\newblock Quasilinear parabolic equations that admit double degeneration.
\newblock \emph{Algebra i Analiz}, 4\penalty0 (6):\penalty0 114--130, 1992.

\bibitem[Ivanov(1994)]{ivanov1994holder}
Aleksandr~Vasil'evich Ivanov.
\newblock H\"older estimates for equations of fast diffusion type.
\newblock \emph{Algebra i Analiz}, 6\penalty0 (4):\penalty0 101--142, 1994.

\bibitem[Ivanov and Mkrtychyan(1991)]{ivanov1991regularity}
AV~Ivanov and PZ~Mkrtychyan.
\newblock On the regularity up to the boundary of generalized solutions of the
  first initial-boundary value problem for quasilinear parabolic equations that
  admit double degeneration.
\newblock \emph{Zap. Nauchn. Sem. Leningrad. Otdel. Mat. Inst. Steklov.(LOMI)},
  196:\penalty0 83--98, 1991.

\bibitem[Kassmann(2008)]{Kassmann2008}
Moritz Kassmann.
\newblock A priori estimates for integro-differential operators with measurable
  kernels.
\newblock \emph{Calculus of Variations and Partial Differential Equations},
  34\penalty0 (1):\penalty0 1--21, April 2008.
\newblock \doi{10.1007/s00526-008-0173-6}.
\newblock URL \url{https://doi.org/10.1007/s00526-008-0173-6}.

\bibitem[Kassmann and Weidner(2023)]{kassmann2023}
Moritz Kassmann and Marvin Weidner.
\newblock The parabolic {Harnack} inequality for nonlocal equations, 2023.
\newblock URL \url{https://arxiv.org/abs/2303.05975}.

\bibitem[Kinnunen and Kuusi(2006)]{Kinnunen2006}
Juha Kinnunen and Tuomo Kuusi.
\newblock Local behaviour of solutions to doubly nonlinear parabolic equations.
\newblock \emph{Mathematische Annalen}, 337\penalty0 (3):\penalty0 705--728,
  November 2006.
\newblock \doi{10.1007/s00208-006-0053-3}.
\newblock URL \url{https://doi.org/10.1007/s00208-006-0053-3}.

\bibitem[Kuusi et~al.(2011)Kuusi, Laleoglu, Siljander, and Urbano]{Kuusi2011}
Tuomo Kuusi, Rojbin Laleoglu, Juhana Siljander, and Jos{\'{e}}~Miguel Urbano.
\newblock H\"{o}lder continuity for {Trudinger's} equation in measure spaces.
\newblock \emph{Calculus of Variations and Partial Differential Equations},
  45\penalty0 (1-2):\penalty0 193--229, October 2011.
\newblock \doi{10.1007/s00526-011-0456-1}.
\newblock URL \url{https://doi.org/10.1007/s00526-011-0456-1}.

\bibitem[Kuusi et~al.(2012)Kuusi, Siljander, and Urbano]{Kuusi2012}
Tuomo Kuusi, Juhana Siljander, and Jose~Miguel Urbano.
\newblock Local {H\"older} continuity for doubly nonlinear parabolic equations.
\newblock \emph{Indiana University Mathematics Journal}, 61\penalty0
  (1):\penalty0 399--430, 2012.
\newblock \doi{10.1512/iumj.2012.61.4513}.
\newblock URL \url{https://doi.org/10.1512/iumj.2012.61.4513}.

\bibitem[Leugering and Mophou(2018)]{Leugering2018}
G\"{u}nter Leugering and Gis{\`{e}}le Mophou.
\newblock Instantaneous optimal control of friction dominated flow in a
  gas-network.
\newblock In \emph{Shape Optimization, Homogenization and Optimal Control},
  pages 75--88. Springer International Publishing, 2018.
\newblock \doi{10.1007/978-3-319-90469-6\_5}.
\newblock URL \url{https://doi.org/10.1007/978-3-319-90469-6\_5}.

\bibitem[Liao(2022)]{liaoholder}
Naian Liao.
\newblock H\"{o}lder regularity for parabolic fractional $p$-{Laplacian}, 2022.
\newblock URL \url{https://arxiv.org/abs/2205.10111}.

\bibitem[Liao and Sch\"{a}tzler(2021)]{liaopart3}
Naian Liao and Leah Sch\"{a}tzler.
\newblock On the {H\"{o}lder} regularity of signed solutions to a doubly
  nonlinear equation. part {III}.
\newblock \emph{International Mathematics Research Notices}, 2022\penalty0
  (3):\penalty0 2376--2400, December 2021.
\newblock \doi{10.1093/imrn/rnab339}.
\newblock URL \url{https://doi.org/10.1093/imrn/rnab339}.

\bibitem[Lions(1972)]{lions1972inequations}
Jacques~Louis Lions.
\newblock \emph{Les in{\'e}quations en m{\'e}canique et en physique}.
\newblock Dunod, 1972.

\bibitem[Mahaffy(1976)]{Mahaffy1976}
M.~W. Mahaffy.
\newblock A three-dimensional numerical model of ice sheets: Tests on the
  barnes ice cap, northwest territories.
\newblock \emph{Journal of Geophysical Research}, 81\penalty0 (6):\penalty0
  1059--1066, February 1976.
\newblock \doi{10.1029/jc081i006p01059}.
\newblock URL \url{https://doi.org/10.1029/jc081i006p01059}.

\bibitem[Severini(1910)]{severini1910sulle}
Carlo Severini.
\newblock Sulle successioni di funzioni ortogonali.
\newblock \emph{Atti dell'Accademia Gioenia}, 3 (5) Memoria XIII, 1a, 1910.

\bibitem[Str\"{o}mqvist(2019{\natexlab{a}})]{Strmqvist2019}
Martin Str\"{o}mqvist.
\newblock Harnack{\textquotesingle}s inequality for parabolic nonlocal
  equations.
\newblock \emph{Annales de l{\textquotesingle}Institut Henri Poincar{\'{e}} C,
  Analyse non lin{\'{e}}aire}, 36\penalty0 (6):\penalty0 1709--1745, October
  2019{\natexlab{a}}.
\newblock \doi{10.1016/j.anihpc.2019.03.003}.
\newblock URL \url{https://doi.org/10.1016/j.anihpc.2019.03.003}.

\bibitem[Str\"{o}mqvist(2019{\natexlab{b}})]{Strmqvist2019bdd}
Martin Str\"{o}mqvist.
\newblock Local boundedness of solutions to non-local parabolic equations
  modeled on the fractional p-{Laplacian}.
\newblock \emph{Journal of Differential Equations}, 266\penalty0 (12):\penalty0
  7948--7979, June 2019{\natexlab{b}}.
\newblock \doi{10.1016/j.jde.2018.12.021}.
\newblock URL \url{https://doi.org/10.1016/j.jde.2018.12.021}.

\bibitem[Trudinger(1968)]{Trudinger1968}
Neil~S. Trudinger.
\newblock Pointwise estimates and quasilinear parabolic equations.
\newblock \emph{Communications on Pure and Applied Mathematics}, 21\penalty0
  (3):\penalty0 205--226, May 1968.
\newblock \doi{10.1002/cpa.3160210302}.
\newblock URL \url{https://doi.org/10.1002/cpa.3160210302}.

\bibitem[Vespri(1992)]{vespri1992local}
Vincenzo Vespri.
\newblock On the local behaviour of solutions of a certain class of doubly
  nonlinear parabolic equations.
\newblock \emph{manuscripta mathematica}, 75\penalty0 (1):\penalty0 65--80,
  1992.

\bibitem[Vespri and Vestberg(2022)]{vespri2022extensive}
Vincenzo Vespri and Matias Vestberg.
\newblock An extensive study of the regularity of solutions to doubly singular
  equations.
\newblock \emph{Advances in Calculus of Variations}, 15\penalty0 (3):\penalty0
  435--473, 2022.

\end{thebibliography}
\end{document}